%% file: minimality.tex
\documentclass[11pt]{amsart}
\usepackage{amsmath}
\usepackage{amsfonts}
\usepackage{amssymb}
\usepackage{amsthm}

\usepackage{enumerate}
\usepackage{tikz}
\usepackage{color, comment}
\usepackage[colorlinks,citecolor=blue, 
backref=page]{hyperref}

\usepackage[utf8]{inputenc}
\usepackage[square,sort,comma,numbers]{natbib}
\usepackage{esint}
\usepackage{bbm}
\usepackage{cancel}
\usepackage{subcaption}
\usepackage{listofitems}
\usepackage{pgfplots}
\pgfplotsset{compat=1.15}
\usepackage{mathrsfs}
\usetikzlibrary{arrows}
\usepackage[a4paper,left=35mm,right=35mm,top=35mm,bottom=35mm,marginpar=25mm]{geometry}

\usepackage{import}
\usepackage{xifthen}
\usepackage{pdfpages}
\usepackage{transparent}

\definecolor{ccqqqq}{rgb}{0.8,0.,0.}
\definecolor{qqqqff}{rgb}{0.,0.,1.}
\definecolor{xfqqff}{rgb}{0.4980392156862745,0.,1.}

\numberwithin{equation}{section}
\newcommand{\eps}{\varepsilon}
\renewcommand{\H}{\mathcal{H}}

\renewcommand{\S}{\mathbb{S}}

\newcommand{\R}{\mathbb{R}}

\newcommand*{\genbf}[1]{\ifmmode\mathbf{#1}\else\textbf{#1}\fi}

\newcommand{\curr}[1]{[\![ #1 ]\!]}

\newcommand{\Lip}{\mathrm{Lip}}

\newcommand{\Irm}{\mathrm{I}}

\newcommand{\Mbf}{\mathbf{M}}
\newcommand{\Tan}{\mathrm{Tan}\,}
\newcommand{\set}{\mathrm{set}}

\def\Gr{\mathcal{G}_n(r_0,\gamma,D, \alpha)}
\newcommand{\Depth}{\operatorname{Depth}}
\newcommand{\diam}{\operatorname{diam}}
\newcommand{\vol}{\operatorname{vol}}

\usepackage{pict2e}
\makeatletter
\DeclareRobustCommand{\intprod}{%
  \mathbin{\mathpalette\int@prod{(0.1,0)(0.9,0)(0.9,0.8)}}}
\DeclareRobustCommand{\restrict}{%
  \mathbin{\mathpalette\int@prod{(0.1,0.7)(0.1,0.1)(0.7,0.1)}}\!}	
\newcommand{\int@prod}[2]{%
  \begingroup
  \sbox\z@{$\m@th#1+$}%
  \setlength\unitlength{\wd\z@}%
  \begin{picture}(1,1)
  \roundcap
  \polyline#2
  \end{picture}%
  \endgroup
}
\makeatother

\newtheorem{theorem}{Theorem}[section]
\newtheorem{lemma}[theorem]{Lemma}
\newtheorem{proposition}[theorem]{Proposition}

\theoremstyle{definition}
\newtheorem{definition}[theorem]{Definition}
\newtheorem{remark}[theorem]{Remark}
\newtheorem{example}[theorem]{Example}
\newtheorem*{problem*}{Problem}

\date{\today}

\author{Giacomo Del Nin}
\address[G. Del Nin]{Max Planck Institute for Mathematics in the Sciences, Inselstrasse 22, 04103 Leipzig, Germany.}
\email{\href{mailto:delnin@mis.mpg.de}{giacomo.delnin@mis.mpg.de}}
\author{Raquel Perales}
\address[R. Perales]{CONACyT Research Fellow at the Math Institute of the 
National Autonomous University of Mexico, Oaxaca. Mexico}
\email{\href{mailto:raquel.perales@im.unam.mx}{raquel.perales@im.unam.mx}}

\title[Lipschitz-mass rigidity]{Rigidity of mass-preserving $1$-Lipschitz maps from integral current spaces into $\mathbb{R}^n$}
\date{\today}

\begin{document}
\begin{abstract}
 We prove that given an $n$-dimensional integral current space and a $1$-Lipschitz map, from this space onto the $n$-dimensional Euclidean ball, that preserves the mass of the current and is injective on the boundary, then the map has to be an isometry. We deduce as a consequence a stability result with respect to the intrinsic flat distance, which implies the stability of the positive mass theorem for graphical manifolds as originally formulated by Huang--Lee--Sormani.
\end{abstract}
\maketitle

\section{Introduction}

Integral currents and flat distance are classical notions from Geometric Measure Theory, being employed, for instance,  by Federer--Fleming to solve the Plateau problem  \cite{FF}. Ambrosio--Kirchheim \cite{AK} defined integral currents in metric measure spaces and applying this work  
Sormani--Wenger \cite{SW2} defined integral current spaces and the intrinsic flat distance.
The intrinsic flat distance between two compact oriented Riemannian manifolds of the same dimension $M_i$, endowed with their canonical currents $\curr{M_i}$, is defined as the infimum of the flat distances between the push-forwards of both currents, 
\[
\inf \Big \{d_{F}(\varphi_{1\sharp}\curr{M_1}, \varphi_{2\sharp}\curr{M_2})\Big \},
\]
where the infimum runs over all complete metric spaces $Z$ and distance preserving embeddings $\varphi_i: M_i \to Z$.

Gromov suggested using intrinsic flat convergence to study manifolds of nonnegative scalar curvature \cite{gromov2014}. In addition, Sormani presented open problems and examples, concerning these type of manifolds, emphasizing the suitability of intrinsic flat convergence \cite{sormani2017}. The intrinsic flat distance has shown to be an adequate notion to study several stability problems, such as the stability of the positive mass theorem and the stability of tori with almost non-negative scalar curvature. Sakovich--Sormani \cite{SS} obtained an intrinsic flat stability result for the positive mass theorem for complete rotationally symmetric asymptotically hyperbolic manifolds. Huang--Lee \cite{HL} showed stability of the positive mass theorem with respect to the Federer--Fleming flat distance for a class of asymptotically flat graphical manifolds.  Motivated by this work,  Huang--Lee--Sormani \cite{HLSa} studied the stability of the positive mass theorem with respect to the intrinsic flat distance. The proof of their first result, \cite[{Theorem 1.3}]{HLSa}, consisted in reducing it to the following rigidity of mass-preserving $1$-Lipschitz maps into Euclidean space.

\begin{problem*}
Let $R>0$ and $(\Omega_R,d_R,T_R)$ be an $n$-dimensional integral current space, and let $\psi:(\Omega_R,d_R)\to(\R^{n+1}, d_{Euc})$ be a $1$-Lipschitz function with the following properties:
\begin{enumerate}
    \item $\psi_\sharp T_R=\curr{B_R \times \{0\}}$, where $B_R$ denotes a ball of radius $R$ in the $n$-dimensional Euclidean space 
    \item $\Mbf(T_R)=\Mbf(\psi_\sharp T_R)=\omega_n R^n$
    \item $\psi|_{\set(\partial T_R)}:\set(\partial T_R) \to \R^{n+1}$ is bi-Lipschitz onto its image, $\partial B_R \times \{0\}$.
    \end{enumerate}
Then $\psi$ should be an isometry between $(\Omega_R,d_R)$ and $(B_R,d_{Euc})$. Hence, as integral current spaces, $(\Omega_R,d_R,T_R)= (B_R,d_{Euc},\curr{B_R})$. 
\end{problem*}

When dealing with Euclidean $n$-currents in $\R^n$ associated with connected open sets, and if we further assume that $\psi$ is $C^1$, a positive answer to the Problem is a consequence of the area formula and classical rigidity for maps with gradient in $SO(n)$.  For completeness, and since the metric case uses similar ideas, we report the proof at the beginning of Section \ref{sec:rigidity}.

A positive answer to the Problem was assumed 
by Huang--Lee-Sormani \cite[{Proof of Theorem 1.3}]{HLSa}  without providing a proof or reference, see the corrigendum written by them \cite{HLSb}, and so the main contribution of this manuscript is to provide a proof of this Problem.

\begin{theorem}[Rigidity]\label{thm-rigidity}
Let $(X,d,T)$ be an $n$-dimensional integral current space, and let $\psi:X\to\R^n$ be a $1$-Lipschitz function with the following properties:
\begin{enumerate}
    \item $\psi_\sharp T=\curr{B_1}$;
    \item $\Mbf(T)=\Mbf(\psi_\sharp T)$;
    \item $\psi$ is injective on $\set(\partial T)$, and $\psi(\set(\partial T))\subseteq \partial B_1$.
\end{enumerate}
Then $\psi$ is an isometry between $(X,d)$ and  $(B_1,d_{Euc})$. Hence, as integral current spaces, $(X,d,T)=(B_1,d_{Euc},\curr{B_1})$.
\end{theorem}

We note that assumption (3), or a similar one, is really necessary to rule out several counterexamples, see Example \ref{ex-giacomo}.

There are other rigidity results by Cecchini--Hanke--Schick \cite{CHS}, Besson--Courtois-Gallot \cite{BCG}, Li \cite{Li}, Li--Wang \cite{LiWang}, Connell--Dai--N\'{u}ñez-Zimbr\'on--Perales--Su\'arez-Serrato--Wei \cite{CDNPSW} of volume preserving $1$-Lipschitz functions defined on spaces with no boundary that satisfy lower curvature bounds, such as scalar curvature bounds, sectional curvature bounds, Alexandrov spaces, Ricci limits and RCD spaces, respectively.  Rigidity results for spaces with boundary have been obtained by Burago--Ivanov \cite{BIhyp,BIeuc} where boundary rigidity and minimal fillings are studied. See also the references within. Furthermore, a similar result to Theorem \ref{thm-rigidity} has been recently proven by Basso, Creutz and Soultanis \cite{BCS}, of which we became aware while in the final stages of the completion of this manuscript. There, condition (3) is replaced by 
$\Mbf( \partial T)=\Mbf(\partial \curr{B_1})$. They include interesting examples of the necessity of the hypotheses. Furthermore, in their work $B_1$ can be taken to be any convex set in Euclidean space.  After this, we realized that Theorem \ref{thm-rigidity} also holds for any convex set in Euclidean space.

Theorem \ref{thm-rigidity} directly implies a positive answer to the Problem above, and therefore the original proof of \cite[{Theorem 1.3}]{HLSa} is valid. See Section \ref{sec-applications} for a sketch of this proof. We note that there is a different and rigorous  proof of  \cite[{Theorem 1.3}]{HLSa} provided
by Huang, Lee and the second named author 
\cite[{Theorem 3.2}]{HLP} that does not rely on the existence of such $1$-Lipschitz function nor on solving the Problem above. The proof uses an intrinsic flat compactness result, \cite[{Theorem 3.4}]{HLP}, which is an easy corollary of \cite[{Theorem 4.2}]{AP} of Allen and the second named author. The latter extends the work of both of them and Sormani to manifolds with boundary. One can also find  a proof of \cite[{Theorem 1.3}]{HLSa} in \cite[{Section 7}]{AP} that applies \cite[{Theorem 4.2}]{AP}, under the added assumption that the manifolds are entire. To obtain a full proof of \cite[{Theorem 1.3}]{HLSa}, in \cite{HLP} the manifolds with non-empty inner boundary are extended to manifolds homeomerphic to balls in Euclidean space and the homeomorphisms are carefully constructed to ensure they are $C^1$, so that one can apply \cite[{Theorem 3.4}]{HLP}.

From Theorem \ref{thm-rigidity} we derive a stability property for the Plateau problem.

\begin{theorem}[Stability]\label{thm-stability}
Let $(X_j, d_j, T_j)$ be a sequence of $n$-dimensional integral current spaces that converge in the intrinsic flat sense to $(X,d,T)$.
Assume that a sequence of $1$-Lipschitz maps $\psi_j: X_j \to \R^{n+1}$ is given,  $\bigcup_{j=1}^{\infty} \psi_j(X_j)$ is contained in a compact set, 
and let $\psi: (X,d) \to \R^{n+1}$ be an Arzel\`{a}-Ascoli intrinsic flat limit of such $\psi_j's$. 
If 
\begin{enumerate}
 \item  $\psi_{\sharp} (\partial T)= \curr{\partial B_R \times \{0\}}$
    \item $\liminf_{j \to \infty}\Mbf(T_j) \leq \Mbf(\curr{B_R \times \{0\}})$
    \item  $\psi$ is injective in  $\set(\partial T)$ and $\psi(\partial T) \subset \partial B_R \times \{0\}$.
    \end{enumerate} 
Then, $(X,d,T)$ equals $(B_R,d_{Euc},\curr{B_R})$.
\end{theorem}

We remark that under the conditions of Theorem \ref{thm-stability} $(X,d)$ could have a different topology than each 
$(X_j, d_j)$ (see Remark \ref{ex-diffTop}). 
Additionally, Theorem \ref{thm-stability} is stronger 
 than \cite[{Theorem 4.2}]{AP} when the limit space is expected to be $(B_1, d_{Euc}, \curr{B_R})$. On the other hand \cite[Theorem 4.2]{AP} allows limit spaces to be different to $(B_R,d_{Euc},\curr{B_R})$. See Remark \ref{ex-APvsDNP} for more details.

Since Theorem \ref{thm-rigidity} also holds for any convex set in Euclidean space, Theorem \ref{thm-stability} also holds for convex sets and this implies that the stability with respect to the intrinsic flat distance of fundamental domains of $3$-dimensional tori with scalar curvature converging to zero, as originally proven (but unpublished) by Cabrera Pacheco, Ketterer and the second named author \cite[{Theorem 5.5}]{CKP-arxive}, holds. This was used to prove stability of $3$-dimensional graphical tori in \cite[{Theorem 1.4}]{CKP-arxive}.  We remark that there is a rigorous and published version of the latter \cite{CPKP}(c.f. \cite[{Theorem 1.4}]{CKP-arxive2}), with a proof in the same spirit as the proof of the stability of the positive mass theorem for entire manifolds that appears in \cite[{Section 7}]{AP}.

We now give a brief outline of the paper. In Section \ref{sec:prelim} we recall some preliminaries on currents and integral current spaces. In Section \ref{sec:rigidity} we prove the rigidity result of Theorem \ref{thm-rigidity}. In Section \ref{sec-examples} we provide an example showing that condition (3) in Theorem \ref{thm-rigidity} cannot be dropped and another one showing that in Theorem \ref{thm-stability} the topology of the sequence can be different to the limit space. Finally, in Section \ref{sec-applications} we prove Theorem \ref{thm-stability} and discuss as an application the stability of the positive mass theorem stated by Huang--Lee--Sormani.

\subsection*{Acknowledgements}

GDN received funding from the European Research Council (ERC) under the European Union's Horizon 2020 research and innovation programme, grant agreement No 757254 (SINGULARITY).
RP acknowledges support from CONACyT Ciencia de Frontera 2019  CF217392 grant.

\section{Metric currents and Integral current spaces}\label{sec:prelim}

We recall that Ambrosio and Kirchheim developed the theory of currents in complete metric spaces \cite{AK} and that Sormani--Wenger developed the notion of integral current spaces and intrinsic flat distance \cite{SW2}.
We assume the reader to be familiar with these works. Nonetheless, we
will briefly define some concepts and results needed in the proofs of our main theorems. 

We denote by $B=B(0,1)$ the open Euclidean unit ball in $\R^n$, and by $B_r$ or $B(r)$ the open ball of radius $r$. Given a complete metric space $X$, we denote by $\Irm_n(X)$ the space of all $n$-dimensional integral currents in $X$.

\subsection{Currents in Euclidean space}
Given an $\H^k$-rectifiable set $E \subset \R^n$, a simple unit $k$-vector field $\tau$ on $E$, and a multiplicity function $\theta:E\to \R$, we denote by $[E,\tau,\theta]$ the Euclidean $k$-current given by
\[
[E,\tau,\theta](\omega)=\int_E \langle\omega(x),\tau(x)\rangle \theta(x) d\H^k(x),
\]
for every $\omega$ smooth, compactly supported $k$-form in $\R^n$. We will denote for simplicity $[B,\tau_{std},1]$ (where $\tau_{std}=e_1\wedge\ldots\wedge e_n$ is the standard orientation of $\R^n$) by $\curr{B}$.
  
Given a Lipschitz map $f:\R^n\to\R^n$, and an $n$-current $T=[E,\tau_{std},\theta]$, 
we recall that the pushforward of $T$ by $f$, denoted by $f_\sharp T$, equals $[f(E),\tau_{std},\tilde\theta]$, where
\begin{equation}\label{eq:pushforward_formula}
    \tilde\theta(y)=\sum_{z\in f^{-1}(y)} \mathrm{sign}(\det(df|_{z})) \theta(z)\qquad\text{for $\H^n$-a.e. $y$}.
\end{equation}
It is a consequence of the area formula that the expression above is well-defined, since the cardinality of $f^{-1}(y)$ is finite for $\H^n$-a.e. $y$, and $f$ is differentiable at $\H^n$-a.e. point by Rademacher's theorem.

\subsection{Structure of rectifiable currents}

Given two complete metric spaces $X,Y$, a Lipschitz function $f:X\to Y$, and an $n$-current $T$ in $X$, the pushforward $f_\sharp T$ is a well-defined $n$-current in $Y$ \cite[Definition 2.4]{AK}, and we have $\Mbf(f_\sharp T)\leq \mathrm{Lip}(f)^n \Mbf(T)$, where $\Mbf(T)$ denotes the mass of the current $T$.

We have the following structure theorem for rectifiable currents in metric spaces, which we can also take as a definition.

\begin{theorem}[{\cite[Theorem 4.5]{AK}}]\label{thm:parametric_representation}
Every rectifiable $n$-current $T$ in a complete metric space $(X,d)$ can be represented as
\begin{equation}\label{eq:decomposition_rectifiable}
T= \sum_{i}(f_i)_{\sharp} \curr{\theta_i}\qquad\text{with}\quad \Mbf(T)=\sum_i\Mbf((f_i)_\sharp\curr{\theta_i})
\end{equation}
for a countable collection of bi-Lipischitz maps $f_i:K_i\to X$, with $K_i$ compact in $\R^n$, $f_i(K_i)$ pairwise disjoint and $\theta_i\in L^1(K_i;\R\setminus\{0\})$. If $T$ is integral, then $\theta_i$ are integer-valued.
\end{theorem}

Given a Lipschitz curve $\gamma:[0,1]\to X$, with $(X,d)$ a complete metric space, we denote by $\curr{\gamma}$ the associated integral $1$-current with weight equal to 1. In the case of integral 1-currents we can say more, as we have the following structure result, proved by Bonicatto--Del Nin--Pasqualetto.

\begin{theorem}[Decomposition of integral $1$-currents {\cite[Theorem 5.3]{BDP}}]\label{thm:decomposition}
Every integral $1$-current $T$ in a complete metric space $(X,d)$ can be written as $T=\sum_i\curr{\gamma_i}$, where $\gamma_i$ are either injective Lipschitz curves or injective Lipschitz loops, and so that
\[
\Mbf(T)=\sum_i \Mbf(\curr{\gamma_i}),\qquad \Mbf(\partial T)=\sum_i \Mbf(\partial\curr{\gamma_i}).
\]
\end{theorem}

\subsection{Area factor} For a rectifiable $n$-current $[E,\tau,\theta]$ in a Euclidean space we have the following formula for its total mass:
\[
\Mbf(T)=\int_E |\theta| d\H^n.
\]
However, for a rectifiable current in a metric space an extra factor appears, the so-called area factor $\lambda$, which we will define now. We remark that the following material is taken from Ambrosio--Kirchheim \cite{AK}.

If $V$ is a Banach space of dimension $n$, the area factor of $V$ is defined to be 
\begin{align*}
\lambda_V= \tfrac{2^n} {\omega_n}   \sup\Big\{ \tfrac{  \H^n(B_1)} {\H^n (L(C)) } \, \vert \, & B_1 \subset L(C) \subset V, \, L:  \mathbb R^n \to V \, \text{ linear}\\
&
\text{and $C$ a Euclidean cube} \Big\},
\end{align*}
where $\omega_n$ is the $n$-Lebesgue measure of a ball of radius $1$ in the $n$-dimensional Euclidean space.

Consider now a $w^*$-separable Banach space $Z$ and a countably $\H^n$-rectifiable
subset $E$ of $Z$, i.e. assume that 
\begin{equation*}
\H^n \left( E \backslash \bigcup_{i=1}^\infty f_i(A_i) \right) = 0
\end{equation*}
with $f_i: A_i \subset \R^n \to Z$ Lipschitz functions, $A_i$ Borel sets, $i \in \mathbb N$. The approximate tangent space to $E$ at a point $x$ is defined as 
\begin{equation*}
\Tan^n(E,x) = wd_y f_i(\R^n),
\end{equation*} 
for $y \in A_i$ that satisfies $f_i(y)=x$, $f_i$ metrically and $w^*$-differentiable at $y$, and $J_n(wd_y f_i) > 0$. 
By \cite{AK}, $\Tan^n(E,x)$ is well defined for $\H^n$-almost all $x \in E$.  Then one defines the area factor of $E$, $\lambda: E \to \R$, as the area factor of $\Tan^n(E, \cdot)$.

Given an arbitrary separable, countably $\H^n$-rectifiable subset $E$ of a metric space $X$, we may isometrically embed $E$ into a $w^*$-separable Banach space $Z$, $\iota:  E \to Z$, then for $\H^n$-almost every $x\in E$
the approximate tangent space of $E$ at $x$ can be defined as
\begin{equation*}
\Tan^n(E,x) : = \Tan^n(\iota(E),\iota(x)).
\end{equation*}
We note that $\Tan^n(E,x)$ is uniquely determined $\H^n$-a.e. up to linear isometries. Then, the area factor of $E$, $\lambda: E \to \R$, 
given by $\lambda(x)=\lambda(\iota(x))$ is well-defined $\H^n$-a.e. We remark that, for a $1$-dimensional current, the area factor is always 1.

Now recall that for an $n$-dimensional current, $T$, the canonical set of $T$ is defined as 
\[
\set(T):= \left\{x \in X \; : \; \liminf_{r \to 0} \frac{\|T\|(B(x,r))}{r^n} > 0\right\}.
\]
This is the smallest set among Borel sets in $X$ in which $T$ is concentrated, up to $\H^{n}$-negligible sets \cite[{Theorem 4.6}]{AK}.

We have the following representation of the mass measure.

\begin{theorem}[{\cite[Theorem 9.5]{AK}}]
For a rectifiable $n$-current $T$ in a complete metric space $(X,d)$ the mass can be represented as
\begin{equation}\label{eq:mass_formula}
\|T\|=\lambda\theta\H^n\restrict E,\qquad \Mbf(T)=\int_E \lambda(x)|\theta(x)|d\H^n(x),
\end{equation}
for some $\H^n$-rectifiable set $E$, some Borel $\H^n$-integrable real multiplicity $\theta:E\to \mathbb{R} \setminus \{0\}$, and where $\lambda$ is the area factor. If $T$ is an integral current then $\theta$ assumes integer values. Moreover, writing $T$ as in \eqref{eq:decomposition_rectifiable}, we have
\begin{equation}\label{eq:theta_formula}
\theta(x)=\theta_i(f_i^{-1}(x))\qquad\text{for $\|T\|$-a.e. $x\in f_i(K_i)$},
\end{equation}
and “$\,\|T\|$-a.e.” can be equivalently replaced with “$\,\H^n\restrict \set(T)$-a.e.”.
\end{theorem}

\subsection{Coarea formula}
A similar reasoning as in the previous discussion, that is, considering an embedding into a $w^*$-separable Banach space, allows to define the tangential differential $d^E\psi_x$ of a Lipschitz function $\psi:E\to\R^n$, with $E$ an $\H^n$-rectifiable set in a metric space $(X,d)$. From this, one also has the following version of the Coarea formula \cite[Theorem 9.4]{AK2}: for a Lipschitz map $\psi:X\to \R^n$, and an $\H^n$-rectifiable set $E\subset X$
\begin{equation}\label{eq:coarea_formula}
\int_E \theta(x) \mathbf{C_n}(d^E \psi_x)d\H^n(x)=\int_{\R^n}\bigg(\sum_{\psi{-1}(y)} \theta(x) \bigg)d\H^n(y).
\end{equation}
Here $\mathbf{C}_n(d^E\psi_x)$ is the coarea factor of the tangential differential of $f$ on $E$ \cite[Section 9]{AK2}. We will not give the precise definition of $\mathbf{C}_n$, since we will only use the following inequality: for a $1$-Lipschitz map $\psi:X\to\R^n$ and an $\H^n$-rectifiable set $E\subset X$ we have 
\begin{equation}\label{eq:lambda_versus_Jacobian}
\mathbf{C}_n (d^E\psi_x) \leq \lambda(x)\qquad\text{for $\H^n$-a.e. $x\in X$.}
\end{equation}
This inequality is proven in \cite[Lemma 9.2]{AK} for the Jacobian $\mathbf{J}_n(d^E\psi_x)$, but the version above follows if one observes that $\mathbf{J}_n(d^E\psi_x)=\mathbf{C}_n(d^E\psi_x)$ for an $n$-dimensional set $E$ and a function $\psi$ with values in $\R^n$ (see the discussion above \cite[Eq. (9.2)]{AK2}).

\subsection{Slicing}
We summarise below some properties of the slicing operator, adapted to our situation.

\begin{proposition}
[{\cite[Theorems 5.6, Theorem 5.7 and c.f. proof of Lemma 5.9]{AK}}]\label{prop:slicing}
Let $T\in\Irm_n(X)$, and let $\pi\in \Lip(X,\R^{n-1})$. Then for $z \in  \R^{n-1}$
there exist currents $\langle T,\pi,z\rangle\in\Irm_{1}(X)$ (called slices), satisfying:
\begin{enumerate} 
\item $\langle T,\pi,z\rangle$ is concentrated on $\set(T)\cap\pi^{-1}(z)$.
    \item The following identity between measures holds:
    \[
    \int_{\R^{n-1}} \|\langle T,\pi,z\rangle \|dz=\|T\restrict d\pi\|,
    \]
    and in particular the following inequality holds:
    \[
    \int_{\R^{n-1}} \Mbf(\langle T,\pi,z\rangle)dz\leq \Lip(\pi)^{n-1}\Mbf(T).
    \]
    \item   $\partial \langle T,\pi,z\rangle =(-1)^{n-1} \langle \partial T,\pi,z\rangle$ for $\H^{n-1}$-a.e. $z$.
    \item   Let $\psi:X\to\R^n$ and $p:\R^n\to\R^{n-1}$ be a pair of Lipschitz maps, then
            \[
            \psi_\sharp \langle T, p \circ \psi,z\rangle=\langle \psi_\sharp T,p,z\rangle\qquad\text{for $\H^{n-1}$-a.e. $z\in\R^{n-1}$.}
            \]
\end{enumerate}
\end{proposition}

\subsection{Integral current spaces}

An integral current space, $(X,d,T)$, consists of a metric space, $(X,d)$, with a current, $T \in \Irm_n(\bar X)$, where $\bar X$ is the completion of $X$, and so that $\set(T)=X$. 

Given an oriented compact Riemannian manifold, $(M,g)$, we associate to it the integral current space $(M, d_g, \curr{M})$, where $d_g$ is the length distance induced by $g$ and $\curr{M}$ is the integral current given by a bi-Lipschitz oriented and countable atlas of $M$, and choosing weight $\theta=1$ (c.f. \cite[{Remarks 2.8, 2.38}]{SW2}).  We often abuse notation and write $\curr{M}$ or even $M$ instead of the whole triple. 

The intrinsic flat distance is defined in the class
of precompact integral current spaces of the same dimension up to current preserving isometry, that is, 
\[
d_{\mathcal F}((X_1,d_1, T_1), (X_2,d_2, T_2))=0
\]
if and only if there exists an isometry $\psi:X_1 \to X_2$ such that $\psi_\sharp T_1=T_2$.

We have the following Arzel\`{a}-Ascoli type theorem for integral current spaces, due to Sormani.

\begin{theorem}[{\cite[Theorem 6.1]{S:AA}}]
Let $M_j=(X_j,d_j, T_j)$ be a sequence of integral current spaces, $j \in \mathbb N \cup \{ \infty\}$, such that $M_j$ converges in the intrinsic flat sense to $M_\infty$. Assume there exist $L$-Lipschitz functions $\psi_j: X_j \to W$, $j \in \mathbb N$, where $W$ is a compact metric space. Then there exists a subsequence, $\psi_{j_k}$, that converges to an $L$-Lipschitz function, $\psi_\infty: X_\infty \to W$. That is, there exist 
a subsequence $M_{j_k}$ and isometric embeddings into a complete metric space $Z$, $\varphi_{k}: X_{j_k} \to Z$, such that 
$\varphi_{k \sharp}(T_{j_k})$ converges in flat sense to $\varphi_{\infty \sharp}(T_{j_\infty})$, and $\psi_\infty$ is then given by
$$\psi_\infty(x_\infty)= \lim_{k \to \infty} \psi_{j_k}(x_k)$$
for any $x_\infty \in X_\infty$ and any sequence $x_{j_k}$ such that 
$\varphi_k(x_{j_k})$ converges to $\varphi_\infty(x_\infty)$. 
\end{theorem}

We will say that $\psi_\infty$ is an Arzel\`{a}-Ascoli intrinsic flat limit of $\psi_j$.

\section{Rigidity of mass-preserving 1-Lipschitz maps}\label{sec:rigidity}

We start this section by giving the proof of the rigidity statement of Theorem \ref{thm-rigidity} in the very special case when the current $T$ is a top-dimensional Euclidean current in $\R^n$ associated with a connected open set, and when the map $\psi$ is $C^1$. In the following subsections, we consider the metric setting and prove Theorem \ref{thm-rigidity} in full generality.

\subsection{Rigidity for top-dimensional Euclidean currents}\label{subsec:Euclidean}
Consider a current $T=[E,\tau_{std},\theta]$, with $E\subset\R^n$ a connected open set and $\theta:E\to \mathbb{N}\setminus\{0\}$ an integrable function. Suppose that $\psi:\R^n\to\R^n$ is a $1$-Lipschitz, $C^1$ map with the property that $\psi_\sharp T=\curr{B_1}$ and $\Mbf(\psi_\sharp T)=\Mbf(T)$. We claim that $\psi$ must be an isometry, i.e., an affine map with gradient in $O(n)$. Notice that connectedness here replaces assumption (3) in Theorem \ref{thm-rigidity}.

Using \eqref{eq:pushforward_formula}, coupled with the fact that in our case $\tilde\theta=1$, we deduce the following chain of inequalities:
\begin{align*}
    \Mbf(\psi_\sharp T)&=\int_{\psi( E)} \tilde\theta (y) d\H^n(y)\\
    &=\int_{\psi( E)}\sum_{z\in \psi^{-1}(y)} \mathrm{sign}(\det \nabla \psi(z)) \theta(z)\,d\H^n(y)\\
    &\leq \int_{\psi( E)}\sum_{z\in \psi^{-1}(y)} |\theta(z)|d\H^n(y)\\
    &=\int_E |\theta(z)|\, |\det \nabla \psi (z)| d\H^n(z)\\
    & \leq \int_E |\theta(z)| d\H^n(z)\\
    &=\Mbf(T)=\Mbf(\psi_\sharp T).
\end{align*}
Besides the triangle inequality and $\Lip(\psi)\le 1$, we have used the area formula to pass from the third to the fourth line. It follows that all inequalities are equalities, and in particular $|\det\nabla \psi(z)|=1$ for $\H^n$-a.e. $z\in E$. By continuity of $\nabla\psi$, and since $E$ is connected, $\det\nabla\psi$ has a constant sign on $E$, and we can assume without loss of generality that it is $+1$.
We deduce that, for $\H^n$-a.e. $z\in\R^n$, we have the following conditions:
\[
\begin{cases}
\Lip(\nabla \psi(z))=1\\
\det\nabla\psi(z)=1.
\end{cases}
\]
From this it follows that $\nabla \psi(z)\in SO(n)$ for $\H^n$-a.e. $z\in E$. To show this one can use for instance the $QR$ decomposition of matrices, writing $\nabla\psi(z)=QR$, with $Q$ orthogonal and $R$ upper triangular. The first condition above implies that all the diagonal elements of $R$ have modulus at most 1, but the second condition implies that they are all of modulus 1, and it follows again from the first condition that there are no off-diagonal terms in $R$. Therefore $R$ is a diagonal matrix with $+1$ or $-1$ in the diagonal. Since the determinant is positive, it follows that $\det\nabla\psi\in SO(n)$.
Finally, one concludes invoking the classical Liouville-type result (see, e.g., \cite{Res}) that a Lipschitz map $\psi:E\to\R^n$, with $E\subset \R^n$ open connected set, whose gradient lies in $SO(n)$ almost everywhere must necessarily be an affine map (and thus an isometry).

\bigskip

From now on we will focus on the rigidity in the metric setting and for a merely $1$-Lipschitz function.

\subsection{Isometry on 1-slices}

Given a complete metric space $X$, we denote by $\delta_x$ the $0$-current associated to the Dirac measure at a point $x\in X$, namely, $ \delta_x(f)=f(x)$ for any  Lipschitz and bounded function $f: X \to \R$.

\begin{lemma}\label{lemma:curve_between_points}
Let $T$ be an integral $1$-current in a complete metric space $(X,d)$, such that 
$\partial T=\delta_b-\delta_a$, for some $a,b\in X$. Then:
\begin{enumerate}
    \item $\Mbf(T)\geq d(a,b)$;
    \item If $\Mbf(T)= d(a,b)$ then $T=\curr{\gamma}$ with $\gamma$ a geodesic between $a$ and $b$.
\end{enumerate} 
\end{lemma}

\begin{proof}
(1) From the hypotheses and the Decomposition of integral $1$-currents's Theorem \ref{thm:decomposition}, we know that 
$$T=\curr{\gamma} + \sum_i\curr{\gamma_i},$$ where $\gamma:[0,1]\to X$  is an injective Lipschitz curve with $\gamma(0)=a$ and $\gamma(1)=b$, and $\gamma_i$ are at most countably many simple Lipschitz loops. Observe that there is only one non-closed curve in the decomposition, since the boundary of $T$ has mass 2. Moreover,
\begin{equation}\label{eq:mass_bigger_than_distance}
\Mbf(T)=\Mbf(\curr{\gamma})+\sum_i \Mbf(\curr{\gamma_i})\geq \Mbf(\curr{\gamma})=\ell(\gamma)\geq d(a,b).
\end{equation}
(2) If $\Mbf(T)=d(a,b)$ then every inequality in \eqref{eq:mass_bigger_than_distance} is an equality. Therefore $T=\curr{\gamma}$, and $\gamma:[0,1]\to X$ is a Lipschitz injective curve with $\ell(\gamma)=d(\gamma(0),\gamma(1))$. This implies that $\gamma$ is a geodesic.
\end{proof}

\begin{lemma}[Isometry on $1$-dimensional slices]\label{lemma:isometry_on_slices}
Let $(X,d,T)$ be an $n$-dimensional integral current space, and let $\psi:X\to \R^n$ be a $1$-Lipschitz map satisfying assumptions (1),(2),(3) of Theorem \ref{thm-rigidity}. Let $v\in\S^{n-1}$ and consider the orthogonal projection $p:\R^n\to v^\perp$. Define the slices
\[ 
T_z:=\langle T,\pi,z\rangle\in \Irm_1(X), \qquad \pi:=p\circ\psi
\] 
 for $\H^{n-1}$-a.e. $z\in D:=B\cap v^\perp$. Then, for $\H^{n-1}$-a.e. $z\in D$, $T_z= \curr{\eta_z}$ where $\eta_z$ is a geodesic and $\psi$ is an isometry between $\set(T_z)$ and the Euclidean segment $B\cap p^{-1}(z)$.
\end{lemma}

We remark that this is the key lemma where we use that the target space is Euclidean. More precisely, we use that the ball can be foliated by geodesics, in such a way that these geodesics are the slices with respect to some $1$-Lipschitz map of the current associated to the ball. In other words, these slices realize the equality in the second equation of item (2) in Proposition \ref{prop:slicing}.

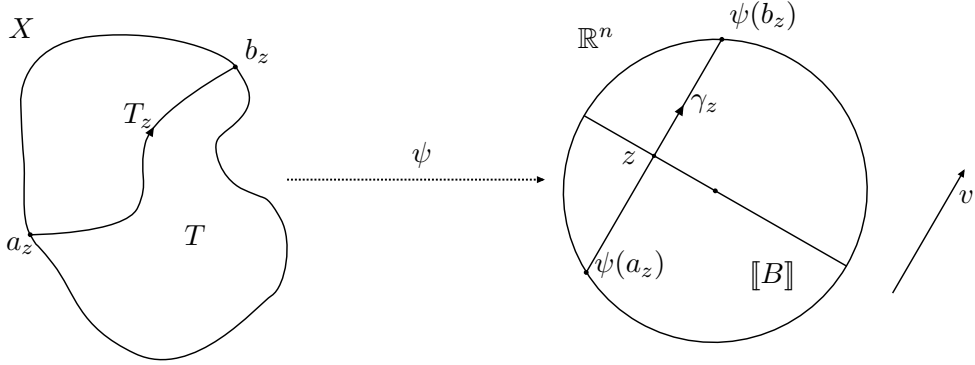
\begin{figure}
    \centering
    \def\svgwidth{\columnwidth}
    \import{./}{slicing.pdf_tex}

    \caption{In the proof of Lemma \ref{lemma:isometry_on_slices} we slice the current $\psi_\sharp T$ with respect to the projection map $p: \R^n \to v^\perp$, obtaining currents $\curr{\gamma_z} \in \Irm_1(\R^n)$ supported in segments perpendicular to $v^\perp$ passing through $z$ and contained in $B$. On the other hand, we slice $T$ with respect to 
    $\pi=p\circ \psi$ obtaining currents $T_z \in \Irm_1(X)$. 
    Note that $\psi_\sharp T_z=\curr{\gamma_z}$. 
    By slicing and the coarea formula we conclude that $\Mbf(T_z)=\Mbf(\curr{\gamma_z})$. Using condition (3) of Theorem \ref{thm-rigidity} and Lemma \ref{lemma:curve_between_points}, we conclude that there are exactly two points $a_z, b_z \in \set(\partial T_z)$ and that $T_z= \curr{\eta_z}$ with $\eta_z$ a geodesic connecting $a_z$ and $b_z$. From this, it easily follows that $\psi$ is an isometry between $\set(T_z)$ and $\set(\curr{\gamma_z})=B\cap \pi^{-1}(z)$.}
    \label{fig:slicing}
\end{figure}

\begin{proof}
For every $z\in D$, let
 $\gamma_z$ be the Euclidean segment $B\cap p^{-1}(z)$ with the orientation parallel to $v$. Note that up to a sign, 
 \[
\curr{\gamma_z}= \langle \curr{B}, p, z \rangle = \langle \psi_\sharp T, p, z \rangle,
 \]
where we used that $\psi_\sharp T=\curr{B}$. By item (4) of Proposition \ref{prop:slicing}, we have (up to a sign)
\[
\psi_\sharp T_z = \psi_\sharp  \langle T,\pi,z\rangle\ =
 \langle \psi_\sharp T,p,z \rangle 
\]
for $\H^{n-1}$-a.e. $z\in D$. Thus, $\psi_\sharp T_z =\curr{\gamma_z}$
and since $\psi$ is $1$-Lipschitz we get 
\[
\Mbf(\curr{\gamma_z})\leq \Mbf(T_z)\qquad\text{for $\H^{n-1}$-a.e. $z\in D$}.
\]
Now applying item (2) of Proposition \ref{prop:slicing} to the slices of $T$ with respect to $\pi$, $T_z$, and the slices of 
$\psi_\sharp T$ with respect to $p$, $\curr{\gamma_z}$, we get
\begin{align*}
\int_{D}\Mbf(T_z)d\H^{n-1}(z)& =  \Mbf(T\restrict d\pi)    \\
& \leq  \Mbf(T)=\Mbf(\psi_\sharp T)\\
&=\int_D \Mbf(\curr{\gamma_z})d\H^{n-1}(z)\\
&\leq \int_D \Mbf(T_z)d\H^{n-1}(z).
\end{align*}
Therefore all inequalities are equalities, and specifically 
\[
\Mbf(T_z)=\Mbf(\curr{\gamma_z})\qquad\text{for $\H^{n-1}$-a.e. $z\in D$}.
\]

We now claim that
\begin{equation*}
\Mbf(\partial T_z)=2\qquad\text{for $\H^{n-1}$-a.e. $z\in D$.}  
\end{equation*}
Since $\psi_\sharp T_z= \curr{\gamma_z}$, the inequality $\Mbf(\partial T_z)\geq \Mbf(\partial \curr{\gamma_z})= 2$ is immediate from the $1$-Lipschitz property of $\psi$.

To prove the other inequality, by item (3) of Proposition \ref{prop:slicing} for $\H^{n-1}$-a.e. $z\in D$ we have
\[
\partial T_z=\partial \langle T,\pi,z\rangle=-\langle \partial T,\pi,z\rangle,
\]
so that by item (1) of Proposition \ref{prop:slicing} we infer that $\partial T_z$ is supported on $\set(\partial T)\cap \pi^{-1}(z)$. Since $\pi^{-1}(z)=\psi^{-1}(p^{-1}(z))$, and $p^{-1}(z)\cap \partial B$ has cardinality at most 2, from assumption (3) of Theorem \ref{thm-rigidity} we deduce that $\pi^{-1}(z)$ has at most 2 points (denote them by $a_z$ and $b_z$). Since we know that $\partial\curr{\gamma_z}=\delta_{\psi(b_z)} - \delta_{\psi(a_z)}$, from the injectivity assumption it follows that $\partial T_z=\delta_{b_z}-\delta_{a_z}$.
This proves  the claim.

We are now in the position to apply Lemma \ref{lemma:curve_between_points}. By the previous paragraphs and the fact that $\psi$ is $1$-Lipschitz, 
\[
\Mbf(T_z)=\Mbf(\curr{\gamma_z})=|\psi(a_z)-\psi(b_z)|\leq d(a_z,b_z),
\]
therefore we conclude that
$d(a_z,b_z)=|\psi(b_z)-\psi(a_z)|$ and that $T_z=\curr{\eta_z}$ where $\eta_z:[0,1] \to X$ is a geodesic between $\eta_z(0)=a_z$ and $\eta_z(1)=b_z$. By the fact that $\eta_z$ is a geodesic, for every $t,s \in [0,1]$, 
\[
|\psi(\eta_z(t))-\psi(\eta_z(s))|=d(\eta_z(t),\eta_z(s))=d(a_z,b_z)|t-s|=
|\psi(a_z)-\psi(b_z)||t-s|,\]
thus we conclude that $\psi$ is an isometry from $\set(T_z)$ to 
$\psi(\set(T_z))=B\cap p^{-1}(z)$. 
\end{proof}

\subsection{Essential injectivity}

We start with the following simple observation that will be used several times throughout.

\begin{lemma}\label{lemma:preserving_subsets}
Let $(X,d)$ and $(Y,\rho)$ be complete metric spaces, and let $\psi:X\to Y$ be a $1$-Lipschitz map. Consider an $n$-current $T$ in $X$ with finite mass, and suppose that $\Mbf(\psi_\sharp T)=\Mbf(T)$. Then for every $\H^n$-measurable set $A\subset X$ we also have $\Mbf(\psi_\sharp( T\restrict A))=\Mbf(T\restrict A)$.
\end{lemma}

\begin{proof} Denote by $A^c:=X\setminus A$. Then
\begin{align*}
\Mbf(\psi_\sharp T)&\leq \Mbf(\psi_\sharp(T\restrict A))+\Mbf(\psi_\sharp(T\restrict A^c))\\
&\leq \Mbf(T\restrict A)+\Mbf(T\restrict A^c)=\Mbf(T),
\end{align*}
where we have used that the mass of every current is not increased by a $1$-Lipschitz map. It follows that all inequalities are equalities, and in particular $\Mbf(\psi_\sharp(T\restrict A))=\Mbf(T\restrict A)$.
\end{proof}

\begin{lemma}[Essential injectivity]\label{lemma-essentialInj}
Suppose that $T$ is an integral $n$-current in a complete metric space $(X,d)$, and $\psi:X\to\R^n$ is a $1$-Lipschitz map such that $\Mbf(\psi_\sharp T)=\Mbf(T)$ and $\psi_\sharp T=\curr{B_1}$. Then:
\begin{enumerate}
    \item There exists an $\H^n$-measurable set $A\subset X$, with $\|T\|(X\setminus A)=0$, such that $\psi$ is injective on $A$;
    \item $T$ has multiplicity $1$ $\|T\|$-almost everywhere.
\end{enumerate}
\end{lemma}

\begin{proof}
By Theorem \ref{thm:parametric_representation} $T$ can be written $T$ as
\[
T=\sum_i (f_i)_\sharp\curr{\theta_i},
\]
for some Lipschitz maps $f_i:K_i\to X$, where $K_i\subset\R^n$ are compact sets and $f_i(K_i)$ are pairwise disjoint, and some multiplicities $\theta_i:K_i\to \mathbb{N}$. Therefore
\[
\psi_\sharp T=\sum_i (\psi\circ f_i)_\sharp \curr{\theta_i}.
\]
Since $\psi\circ f_i:K_i \subset \R^n\to\R^n$, we know from \eqref{eq:pushforward_formula} that each term in the right hand side of the previous equation can be written as
\[
(\psi\circ f_i)_\sharp \curr{\theta_i}=[\tilde E_i,\tau_{std},\tilde\theta_i]
\]
where $\tilde E_i=\psi(f_i(K_i))$, $\tau_{std}$ is the standard orientation of $\R^n$, and
\[
\tilde\theta_i(y)=\sum_{z\in (\psi\circ f_i)^{-1}(y)} \pm \theta_i(z)\qquad\text{for $\H^n$-a.e. $y\in \psi(f_i(K_i))$,}
\]
where the plus or minus sign equals the sign of the Jacobian determinant of $\psi\circ f_i$ at the point $z$. Setting $E=\bigcup_i f_i(K_i)$, by summing over $i$ it follows that
\[
\psi_\sharp T=[\psi(E), \tau_{std},\tilde \theta]
\]
where
\[
\tilde \theta(y)=\sum_{i}\sum_{z\in (\psi\circ f_i)^{-1}(y)} \pm\theta_i(z)
\qquad\text{for $\H^n$-a.e. $y\in \psi(E)$}.
\]
Now, since $f_i$ is a bijection between $K_i$ and $f_i(K_i)$, by just relabeling points we can write, for any fixed $i$,
\begin{align*}
\sum_{z\in(\psi\circ f_i)^{-1}(y)}\pm\theta_i(z)&=\sum_{x\in\psi^{-1}(y)\cap f_i(K_i)}\pm\theta_i(f_i^{-1}(x))\\
&\overset{\eqref{eq:theta_formula}}{=} \sum_{x\in\psi^{-1}(y)\cap f_i(K_i)}\pm\theta(x).
\end{align*}
Summing over $i$ it follows from above that
\begin{align*}
    \tilde \theta(y)&=\sum_{i}\sum_{z\in (\psi\circ f_i)^{-1}(y)} \pm\theta_i(z)\\
    &=\sum_i\sum_{x\in\psi^{-1}(y)\cap f_i(K_i)}\pm\theta(x)\\
    &=\sum_{x\in\psi^{-1}(y)} \pm\theta(x)\qquad\text{for $\H^n$-a.e. $y\in \psi(E)$.}
\end{align*}
By the triangle inequality it follows that
\begin{equation}\label{eq:inequality_multiplicities}
|\tilde\theta(y)|\leq \sum_{x\in\psi^{-1}(y)} |\theta(x)|\qquad\text{for $\H^n$-a.e. $y\in \psi(E)$.}
\end{equation}
From this we derive the following estimate:
\begin{align*}
    \Mbf(\psi_\sharp T)&=\int_{\psi( E)} |\tilde\theta (y)|d\H^n(y)\\
    &\leq \int_{\psi( E)}\sum_{x\in \psi^{-1}(y)} |\theta(x)|d\H^n(y)\\
    &=\int_E |\theta(x)| \mathbf{C}_n(d^E\psi_x) d\H^n(x)\\
    & \leq \int_E |\theta(x)| \lambda (x) d\H^n(x)\\
    &=\Mbf(T)=\Mbf(\psi_\sharp T),
\end{align*}
where in the previous chain of inequalities we have used, in the following order, that the area factor of the Euclidean space equals $1$, \eqref{eq:mass_formula}, \eqref{eq:inequality_multiplicities}, the coarea formula \eqref{eq:coarea_formula}, \eqref{eq:lambda_versus_Jacobian}, \eqref{eq:mass_formula} again, and the mass-preserving hypothesis. It follows that all inequalities are equalities, and in particular
\begin{equation}\label{eq:equality_multiplicities}
1=|\tilde\theta(y)|= \sum_{x\in\psi^{-1}(y)} |\theta(x)|\qquad\text{for $\H^n$-a.e. $y\in B$}
\end{equation}
because we know by assumption that $\psi_\sharp T=\curr{B}=[B,\tau,1]$. Now \eqref{eq:equality_multiplicities} implies both that the cardinality of $\psi^{-1}(y)$ is 1 and that $|\theta(\psi^{-1}(y))|=1$ for all points $y\in G$, for some measurable $G\subset B$ with full $\H^n$-measure. Setting $A:=\psi^{-1}(G)$ we obtain the desired properties. Indeed, $\psi$ is injective on $A$, $|\theta|=1$ on $A$, and moreover $\psi_\sharp(T\restrict A)=(\psi_\sharp T)\restrict G=\psi_\sharp T$, therefore by Lemma \ref{lemma:preserving_subsets} $\Mbf(T\restrict A)=\Mbf(\psi_\sharp(T\restrict A))=\Mbf(\psi_\sharp T)=\Mbf(T)$. It follows that $\|T\|(X\setminus A)=0$.
\end{proof}

\subsection{Proof of the rigidity theorem}

In the proof of the main theorem we will need the following elementary lemma.

\begin{lemma}\label{lemma:Lebesgue}
Let $A_1,A_2$ be subsets of $\R^n$ with positive Lebesgue measure. Then there exists $v\in\mathbb{S}^{n-1}$ such that, denoting by $p_v: \R^n \to \R^n$ the orthogonal projection on $v^\perp$, it holds
\[
\H^{n-1}(p_v(A_1)\cap p_v(A_2))>0.
\]
\end{lemma}

\begin{proof}
Given any two Lebesgue points $a_1\in A_1$ and $a_2\in A_2$, we prove that the conclusion holds with $v$ parallel to $a_2-a_1$. We can assume without loss of generality that $a_1$ is the origin and $a_2$ is the $n$-th basis vector, and thus $v=e_n$. In the Lebesgue density theorem we can equivalently replace balls by cylinders $C_r(a_i)=D_r(p_v(a_i))\times I_r$, where $D_r(\pi_v(a_i))\subset \R^{n-1}$ is a disk of radius $r$ centered at $\pi_v(a_i)$, and $I_r$ is a segment of length $2r$, so that we can assume
\begin{equation}\label{eq:Lebesgue}
\lim_{r\to 0}\frac{|A_i\cap C_r(a_i)|}{|C_r(a_i)|}=1,\qquad i=1,2.
\end{equation}
It follows that, for $i=1,2$
\[
\H^{n-1}(D_r\setminus p_v(A_i))\leq \frac{1}{2r}|C_r(a_i)\setminus A_i|=\H^{n-1}(D_r) \frac{|C_r(a_i)\setminus A_i|}{|C_r(a_i)|}.
\]
The ratio to the right goes to zero as $r\to 0$ by \eqref{eq:Lebesgue}, thus for every $\eps>0$, for sufficiently small $r>0$
\begin{equation}\label{eq:1-eps}
\H^{n-1}(D_r\cap p_v(A_i))\geq (1-\eps) \H^{n-1}(D_r).
\end{equation}
Setting $\mu:=\H^{n-1}\restrict D_r$, by the inclusion-exclusion principle we have
\begin{align*}
    \mu(p_v(A_1)\cap p_v(A_2))& =\mu(p_v(A_1))+\mu(p_v(A_2))-\mu(p_v(A_1)\cup p_v(A_2))\\
    &\ge 2(1-\eps) \H^{n-1}(D_r)-\H^{n-1}(D_r)\\
    &=(1-2\eps) \H^{n-1}(D_r),
\end{align*}
where we used twice \eqref{eq:1-eps}, together with $\mu(p_v(A_1)\cup p_v(A_1))\le \H^{n-1}(D_r)$.
This shows that for $r>0$ small enough $\H^{n-1}(D_r\cap p_v(A_1)\cap p_v(A_2))>0$, which implies the thesis.
\end{proof}

We finally come to the proof of the main theorem.
\begin{proof}[Proof of Theorem \ref{thm-rigidity}]
Take any two density points $x_1,x_2$ of $\set(T)\cap A$, where $A$ is the injectivity set defined in Lemma \ref{lemma-essentialInj}. We are going to show that $d(x_1,x_2)=|\psi(x_1)-\psi(x_2)|$.

Take any $\delta>0$. Since $x_1,x_2$ belong to $\set(T)$ then $T\restrict B(x_i,\delta)$ are non zero for $i=1,2$. From the mass-preserving hypothesis and Lemma \ref{lemma:preserving_subsets} we deduce that also $\psi_\sharp(T\restrict B(x_i,\delta))$ are non zero for $i=1,2$, and moreover from the $1$-Lipschitz property they are supported on $B(\psi(x_1),\delta)$ and $B(\psi(x_2),\delta)$ respectively. From the fact that the latter are non-trivial rectifiable $n$-dimensional currents in $\R^n$, their mass is absolutely continuous with respect to Lebesgue measure. This implies that $A_i:=\psi(\set(T)\cap A\cap B(x_i,\delta))$ is a set of positive measure for $i=1,2$. By applying Lemma \ref{lemma:Lebesgue}, we find $v\in \mathbb{S}^{n-1}$ such that $p_v(A_1)\cap p_v(A_2)$ has positive $\H^{n-1}$-measure. Combining this with Lemma \ref{lemma:isometry_on_slices} (with $p=p_v$), we find at least two points $y_1\in B(\psi(x_1),\delta)$ and $y_2\in B(\psi(x_2),\delta)$ such that $\psi$ is an isometry between a geodesic with end-points in $\psi^{-1}(y_1)$ and $\psi^{-1}(y_2)$, and the segment between $y_1$ and $y_2$. We deduce that
\[
d(\psi^{-1}(y_1),\psi^{-1}(y_2))=|y_1-y_2|
\]
and consequently, by the triangle inequality
\[
\big|d(x_1,x_2)-|\psi(x_1)-\psi(x_2)|\big|\leq 4\delta.
\]
Here we used that $y_i\in B(\psi(x_i),\delta)$ and $\psi^{-1}(y_i)\in B(x_i,\delta)$ for $i=1,2$.
From the arbitrariness of $\delta>0$ we deduce that $d(x_1,x_2)=|\psi(x_1)-\psi(x_2)|$ for all $x_1,x_2$ density points of $\set(T)\cap A$. Since the latter is a set of full $\|T\|$-measure, it is dense in the closure of $\set(T)$, therefore the isometry property extends by density to $\mathrm{spt}(T)=\overline{\set(T)}$.
\end{proof}

\section{Examples and remarks}\label{sec-examples}

Here we present two examples discussing the hypotheses required in 
Theorem \ref{thm-rigidity} and Theorem \ref{thm-stability}, and add a remark comparing Theorem \ref{thm-stability} with Theorem 4.2 of Allen and the second named author \cite{AP}.

\begin{example}\label{ex-giacomo}
Assumption (3) in Theorem \ref{thm-rigidity}, or a similar one, is really necessary to rule out
several counterexamples, such as a disconnected space. 
For instance,
let $B_{\pm}$ be the sets given by 
$$B_{-}=\{ (x,y) \in \R^2 \,|\, x \leq -1, \,\, (x+1)^2 + y^2 \leq 1\},$$ 
and 
$$B_{+}=\{ (x,y) \in \R^2 \,|\, x \geq 1, \,\, (x-1)^2 + y^2 \leq 1\}.$$
Define $X = B_{-} \cup B_{+}$, endow this set 
with the restriction of the Euclidean distance, d, and let $T$ be the current $T = \curr{X}$ where we give the standard orientation to both sets $B_{-}$ and $B_{+}$.
Then the map $\psi: X \to B \subset \R^2$ given by 
\begin{equation*}
\psi(x,y)= 
\begin{cases}
(x+1,y) & \text{if } (x,y)\in B_-\\
(x-1,y) & \text{if } (x,y) \in B_+ \\
\end{cases}
\end{equation*}
is a $1$-Lipschitz function. The integral current space $(X,d,T)$ and $\psi$ satisfy all the assumptions in Theorem \ref{thm-rigidity} except item $(3)$, but this map is clearly not an isometry.
\end{example}

\begin{example}\label{ex-diffTop}
Note that in Theorem \ref{thm-stability} $(X,d)$ could have a different topology than each 
$(X_j, d_j)$. Consider, for example, 
$(X_j, d_j, T_j)$ with 
$X_j=\overline{B_R\setminus B_{\epsilon_j}} \subset \R^n$,
$d_j$ the length distance induced by the Euclidean distance and $T_j= \curr{\overline{B_R\setminus B_{\epsilon_j}}}$, where $\epsilon_j \to 0$. 
The intrinsic flat limit of this sequence equals $(\overline{B_R}, d_{Euc},\curr{\overline{B_R}})$, and 
taking $\psi_j$ as the inclusion maps, all the hypotheses in Theorem \ref{thm-stability} are satisfied. Furthermore, $\psi_j({\partial T_j}) \neq \curr{\partial B_R}$. A similar behavior occurs in \cite[{Theorem 1.3}]{HLSa}, c.f. Theorem \ref{thm-HLS}, where the integral current spaces have two disconnected boundary components: the inner boundaries, $\partial M_j$, converge to the zero integral current space, while the outer boundaries converge to an integral current space that is bi-Lipschitz equivalent to $(\partial \overline{B}_R \times \{0\}, d_{Euc}, \curr{\partial \overline{B}_R \times \{0\}})$.
\end{example}

\begin{remark}\label{ex-APvsDNP}
We now compare Theorem \ref{thm-stability}  with \cite[{Theorem 4.2}]{AP}, which states the following.
Let $M$ be a compact and oriented manifold, and
$(M, d_j, T_j)$, $j \in \mathbb N$, be a sequence of integral current spaces with $d_j$ given by a $C^0$ Riemannian metric on $M$, and $T_j$ the current with weight $1$. Let $d_0$ be a distance given by a $C^2$ Riemannian metric on $M$, assume that $(M,d_0)$  is totally convex, and let $T_0$ be the corresponding current with weight $1$. If there exist $\psi_j:(M, d_j) \to (M, d_0)$ $C^1$ bi-Lipschitz maps with $\Lip(\psi_j) \leq 1$, 
$\diam(M,d_j) \leq D$, $\liminf_{j \to \infty}\vol(M,d_j) \leq \vol(M,d_0)$, 
then 
$(M, d_j, T_j)$ converges in the intrinsic flat sense to 
$(M, d_0, T_0)$. 
So, when taking $(M,d_0)=(B_1, d_{Euc})$, Theorem \ref{thm-stability} is stronger since the $(X_j,d_j)$ do not need to be homeomorphic to $(M,d_0)$, and the maps $\psi_j$ only need to be $1$-Lipschitz. On the other hand \cite[Theorem 4.2]{AP} allows more general manifolds as a limit space, and not only the ball.
\end{remark}

\section{Application: stability of the Positive Mass Theorem}\label{sec-applications}
We now apply Theorem \ref{thm-rigidity} to deduce the stability result of Theorem \ref{thm-stability}. We then show how to derive the stability of the positive mass theorem of graphical manifolds of Huang--Lee--Sormani.

\begin{proof}[Proof of Theorem \ref{thm-stability}]
Proceeding as in \cite[Proof of Theorem 1.3]{HLSa}:
\begin{align*}
\vol(B(R))&\le \Mbf(\psi_{\sharp}  T) \\
& \le \Mbf(T)\\
 &\le \liminf_{j\to\infty} \Mbf(T_j)\\
& \le \vol(B(R)),
\end{align*}
where the minimizing property of the disk among Euclidean integral currents with the same boundary is used, the fact that $\Lip(\psi)\le 1$ and lower semicontinuity of mass. We recall that the notion of metric mass for metric rectifiable currents in a Euclidean space coincides with the  notion of mass in the Federer--Fleming sense as a consequence of Lemma 9.2 and Theorem 9.5 in \cite{AK}.
Now, equality in the first step implies that $\psi_{\sharp}T=\curr{ B(R)\times \left\{0\right\}}$ (since the disk is the \textit{unique} mass minimizer \cite{Alm}). Equality in the second inequality, implies that 
$\Mbf(T)=\Mbf(\psi_\sharp T)$.  By applying  Theorem \ref{thm-rigidity}, with a rescaling sending $B(R)$ to $B(1)$, we obtain that $\psi$ must be an isometry, and thus $(X,d,T)$ equals  $(B(R),d_{Euc},\curr{B(R)})$. This concludes the proof.
\end{proof}

The positive mass theorem of Schoen--Yau and Witten~\cite{SY, W} states that any complete asymptotically flat manifold of nonnegative scalar curvature has nonnegative ADM mass, and if the ADM mass is zero then the manifold must be the Euclidean space. 
Here we give some details about the intrinsic flat stability of the positive mass theorem formulated by Huang--Lee--Sormani in \cite{HLSa}.
We first define their class of uniformly asymptotically flat graphical hypersurfaces of $\R^{n+1}$ with uniformly bounded depth and nonnegative scalar curvature.

\begin{definition}
For $n\ge3$, $r_0, \gamma, D>0$, and $\alpha<0$, define $\Gr$ to be the space of all smooth complete Riemannian manifolds with nonnegative scalar curvature, $(M^n,g)$ that admit a smooth Riemannian isometric embedding $\psi:M\to \R^{n+1}$ such that for some open $U\subset B(r_0/2)\subset \R^n$, the image $\psi(M)$ is the graph of a function $f\in C^\infty(\R^{n}\smallsetminus \overline{U}) \cap C^0(\R^{n}\smallsetminus {U})$,
\begin{equation*}
    \psi(M)=\left\{(x,f(x)): \,\, x\in \R^n \smallsetminus U\right\},
\end{equation*}
with empty or minimal boundary, 
that is, either $U=\emptyset$ or $f$ is constant on each component of $\partial U$ and $\lim_{x \to \infty}|D f|(x) \to 0$. Assume that for almost every $h$, the level set $f^{-1}(h)\subset \R^n$ is strictly mean-convex and outward-minimizing,
where strictly mean-convex means that the mean curvature is strictly positive, and outward-minimizing means that any region of $\R^n$ that contains the region enclosed by $f^{-1}(h)$ must have perimeter at least as large as $\mathcal{H}^{n-1}( f^{-1}(h))$. Assume that $f$ satisfies the following uniform asymptotic flatness conditions:
\begin{equation*}
 |D f| \le \gamma
\textrm{ for }|x|\ge r_0/2  \textrm{ and }
\lim_{x\to\infty} |D f| =0. 
\end{equation*}
Assume that 
$f(x)$ approaches a constant as $x\to\infty$. 
If $n=3$ or $4$, additionally assume that the graph is asymptotically Schwarzschild,
\begin{equation*}
\exists \Lambda\in \R \textrm{ such that }
 \left|f(x) - (\Lambda+ S_m(|x|)) \right| \le \gamma |x|^\alpha 
 \textrm{ for } |x|\ge r_0.
\end{equation*}
Finally assume that the regions
\begin{equation*}
\Omega(r_0)=\psi^{-1}(\overline{B(r_0)}\times\R)\,\,\,
\textrm{ and } \,\,\, \Sigma(r_0)=\partial\Omega(r_0) \setminus \partial M
\end{equation*}
have bounded depth
 \begin{equation*}
 \Depth(\Omega(r_0), \Sigma(r_0)) =\sup\left\{d_g(p,\Sigma(r_0)): p\in \Omega(r_0)\right\}
 \le D.
 \end{equation*}
\end{definition}

Above, for any $m>0$, the function $S_m :  \mathbb{R}^n \smallsetminus B((2m)^{1/(n-2)})   \to \R$ is such that its graph corresponds to the Riemannian isometric embedding into $\mathbb{R}^{n+1}$ of one end of the spatial $n$-dimensional Schwarzschild manifold of ADM mass $m>0$  such that its minimal boundary lies in the plane $\R^n\times \left\{0\right\}$.


Huang--Lee--Sormani stated that the preimages of the intersections
of the graph $\psi_j(M_j)$ with the cylinder $\overline{B(r)} \times \R$ converge to $\overline{B(r)}$.

\begin{theorem}\label{thm-HLS}
Let $M_j \in\Gr$ be a sequence of asymptotically flat manifolds 
so that $m_{ADM}(M_j) \to 0$. 
  Then for any $r  \geq r_0$, the sequence $  \Omega_j(r)$ subconverges in the intrinsic flat sense to 
  $ \overline{B(r)}$.
 \end{theorem}
 
Now we review the proof in \cite{HLSa}, adding the details that make use of Theorem \ref{thm-stability}, and thus filling in the gap in \cite{HLSa}.

\begin{proof}
In the proof of Theorem 3.1 in \cite{HLSa} uniform upper bounds for 
$\diam(\Omega_j(r))$, $\vol(\partial \Omega_j(r))$ and $\vol(\Omega_j(r))$,
that only depend on the parameters $n,r_0, \gamma, D, \alpha$ and $r$, were obtained. Hence, it is concluded by applying Wenger's compactness theorem and Sormani's Arzel\`{a}-Ascoli's theorem, that one gets subconvergence with respect to the intrinsic flat distance of $\Omega_j(r)$ to an integral current space $(\Omega_r, d_r, T_r)$
and that a $1$-Lipschitz function $\psi_r : \Omega_r \to \overline{B(r)} \times [-D, D]$ exists, which is the limit of the corresponding subsequence of $1$-Lipschitz functions $\psi_j|_{\Omega_j(r)}:\Omega_j(r) \to \overline{B(r)} \times [-D, D]$.

In Corollary 4.4  in \cite{HLSa}  it is shown that 
\begin{equation*}
\label{limsupVolume}
\limsup_{j \to \infty}  \vol(\Omega_j(r)) = \vol(B(r)).
\end{equation*}
In Lemma 4.5 in \cite{HLSa} it is shown that 
$\psi_r ( \Omega_r ) \subset \overline{B(r)} \times \{0\}$, and in particular, 
$\psi_r ( \set(\partial T_r)) \subset \partial \overline{B(r)} \times \{0\}$. Then in Lemma 3.6 in \cite{HLP} (note that this corresponds to Lemma 5.1 in \cite{HLSa} but it had a gap that was corrected in in \cite{HLP}), it is shown that 
$$\psi_r|_{\set(\partial T_r) }: \set(\partial T_r) \to \partial \overline{B(r)} \times \{0\}$$
is bi-Lipschitz and in particular
$\partial \curr{ B(r)\times \left\{0\right\}}=\psi_{r\sharp}(\partial T_r)$. 
With an obvious rescaling, sending $B(r)$ to $B(1)$, we
can apply Theorem \ref{thm-stability} to obtain that $\psi_r$ must be an isometry, and that the limit integral current space $(\Omega_r,d_r,T_r)$ is isometric to $(B(r),d_{Euc},\curr{B(r)})$. This ends the proof of the stability property.
\end{proof}

\bibliographystyle{plain}
\bibliography{minimality}

\end{document}

%% file: Slicing.pdf_tex
\begingroup%
  \makeatletter%
  \providecommand\color[2][]{%
    \errmessage{(Inkscape) Color is used for the text in Inkscape, but the package 'color.sty' is not loaded}%
    \renewcommand\color[2][]{}%
  }%
  \providecommand\transparent[1]{%
    \errmessage{(Inkscape) Transparency is used (non-zero) for the text in Inkscape, but the package 'transparent.sty' is not loaded}%
    \renewcommand\transparent[1]{}%
  }%
  \providecommand\rotatebox[2]{#2}%
  \newcommand*\fsize{\dimexpr\f@size pt\relax}%
  \newcommand*\lineheight[1]{\fontsize{\fsize}{#1\fsize}\selectfont}%
  \ifx\svgwidth\undefined%
    \setlength{\unitlength}{546.58137284bp}%
    \ifx\svgscale\undefined%
      \relax%
    \else%
      \setlength{\unitlength}{\unitlength * \real{\svgscale}}%
    \fi%
  \else%
    \setlength{\unitlength}{\svgwidth}%
  \fi%
  \global\let\svgwidth\undefined%
  \global\let\svgscale\undefined%
  \makeatother%
  \begin{picture}(1,0.38241882)%
    \lineheight{1}%
    \setlength\tabcolsep{0pt}%
    \put(0,0){\includegraphics[width=\unitlength,page=1]{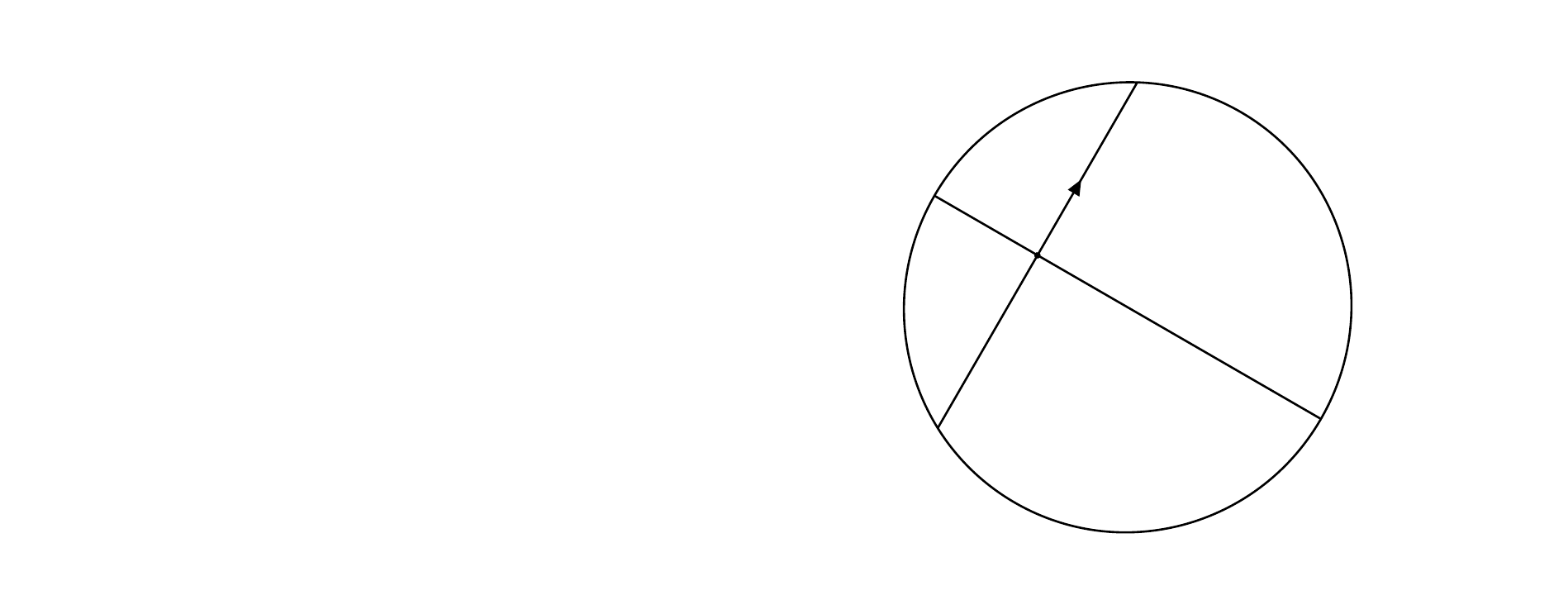}}%
    \put(0.63121563,0.21222842){\color[rgb]{0,0,0}\makebox(0,0)[lt]{\lineheight{1.25}\smash{\begin{tabular}[t]{l}$z$\end{tabular}}}}%
    \put(0.69616444,0.26569186){\color[rgb]{0,0,0}\makebox(0,0)[lt]{\lineheight{1.25}\smash{\begin{tabular}[t]{l}$\gamma_z$\end{tabular}}}}%
    \put(0,0){\includegraphics[width=\unitlength,page=2]{Slicing.pdf}}%
    \put(0.94799808,0.17505333){\makebox(0,0)[lt]{\lineheight{1.25}\smash{\begin{tabular}[t]{l}$v$\end{tabular}}}}%
    \put(0,0){\includegraphics[width=\unitlength,page=3]{Slicing.pdf}}%
    \put(0.43404787,0.21368843){\makebox(0,0)[lt]{\lineheight{1.25}\smash{\begin{tabular}[t]{l}$\psi$\end{tabular}}}}%
    \put(0,0){\includegraphics[width=\unitlength,page=4]{Slicing.pdf}}%
    \put(0.05265022,0.13074265){\transparent{0.99900001}\makebox(0,0)[lt]{\lineheight{1.25}\smash{\begin{tabular}[t]{l}$a_z$\end{tabular}}}}%
    \put(0.27709237,0.31128873){\makebox(0,0)[lt]{\lineheight{1.25}\smash{\begin{tabular}[t]{l}$b_z$\end{tabular}}}}%
    \put(0,0){\includegraphics[width=\unitlength,page=5]{Slicing.pdf}}%
    \put(0.16312583,0.25045495){\makebox(0,0)[lt]{\lineheight{1.25}\smash{\begin{tabular}[t]{l}$T_z$\end{tabular}}}}%
    \put(0,0){\includegraphics[width=\unitlength,page=6]{Slicing.pdf}}%
    \put(0.60558605,0.10929535){\makebox(0,0)[lt]{\lineheight{1.25}\smash{\begin{tabular}[t]{l}$\psi(a_z)$\end{tabular}}}}%
    \put(0.73259749,0.34536785){\makebox(0,0)[lt]{\lineheight{1.25}\smash{\begin{tabular}[t]{l}$\psi(b_z)$\end{tabular}}}}%
    \put(0.59014045,0.32222458){\makebox(0,0)[lt]{\lineheight{1.25}\smash{\begin{tabular}[t]{l}$\mathbb{R}^n$\end{tabular}}}}%
    \put(0.05489164,0.33140814){\makebox(0,0)[lt]{\lineheight{1.25}\smash{\begin{tabular}[t]{l}$X$\end{tabular}}}}%
    \put(0.22033676,0.13532376){\makebox(0,0)[lt]{\lineheight{1.25}\smash{\begin{tabular}[t]{l}$T$\end{tabular}}}}%
    \put(0.75117899,0.09960102){\makebox(0,0)[lt]{\lineheight{1.25}\smash{\begin{tabular}[t]{l}$\curr{B}$\end{tabular}}}}%
    \put(0,0){\includegraphics[width=\unitlength,page=7]{Slicing.pdf}}%
  \end{picture}%
\endgroup%

%% file: minimality.bbl
\begin{thebibliography}{10}

\bibitem{AP}
Brian Allen and Raquel Perales.
\newblock Intrinsic flat stability of manifolds with boundary where volume
  converges and distance is bounded below.
\newblock {\em {A}rxiv: 2006.13030}, 2022.

\bibitem{Alm}
F.~Almgren.
\newblock Optimal isoperimetric inequalities.
\newblock {\em Indiana Univ. Math. J.}, 35(3):451--547, 1986.

\bibitem{AK}
Luigi Ambrosio and Bernd Kirchheim.
\newblock Currents in metric spaces.
\newblock {\em Acta Math.}, 185(1):1--80, 2000.

\bibitem{AK2}
Luigi Ambrosio and Bernd Kirchheim.
\newblock Rectifiable sets in metric and {B}anach spaces.
\newblock {\em Math. Ann.}, 318(3):527--555, 2000.

\bibitem{BCS}
Giuliano Basso, Paul Creutz, and Elefterios Soultanis.
\newblock Filling minimality and lipschitz-volume rigidity of convex bodies
  among integral current spaces.
\newblock {\em ArXiv:2209.12545}, 2022.

\bibitem{BCG}
G.~Besson, G.~Courtois, and S.~Gallot.
\newblock Entropies et rigidit\'{e}s des espaces localement sym\'{e}triques de
  courbure strictement n\'{e}gative.
\newblock {\em Geom. Funct. Anal.}, 5(5):731--799, 1995.

\bibitem{BDP}
Paolo Bonicatto, Giacomo Del~Nin, and Enrico Pasqualetto.
\newblock Decomposition of integral metric currents.
\newblock {\em J. Funct. Anal.}, 282(7):Paper No. 109378, 28, 2022.

\bibitem{BIeuc}
Dmitri Burago and Sergei Ivanov.
\newblock Boundary rigidity and filling volume minimality of metrics close to a
  flat one.
\newblock {\em Ann. of Math. (2)}, 171(2):1183--1211, 2010.

\bibitem{BIhyp}
Dmitri Burago and Sergei Ivanov.
\newblock Area minimizers and boundary rigidity of almost hyperbolic metrics.
\newblock {\em Duke Math. J.}, 162(7):1205--1248, 2013.

\bibitem{CKP-arxive}
Armando~J. Cabrera~Pacheco, Christian Ketterer, and Raquel Perales.
\newblock Stability of graphical tori with almost nonnegative scalar curvature.
\newblock {\em arXiv:1902.03458v1}, 2019.

\bibitem{CPKP}
Armando~J. Cabrera~Pacheco, Christian Ketterer, and Raquel Perales.
\newblock Stability of graphical tori with almost nonnegative scalar curvature.
\newblock {\em Calc. Var. Partial Differential Equations}, 59(4):Paper No. 134,
  27, 2020.

\bibitem{CKP-arxive2}
Armando~J. Cabrera~Pacheco, Christian Ketterer, and Raquel Perales.
\newblock Stability of graphical tori with almost nonnegative scalar curvature.
\newblock {\em arXiv:1902.03458v2}, 2020.

\bibitem{CHS}
Simone Cecchini, Bernhard Hanke, and Thomas Schick.
\newblock Lipschitz rigidity for scalar curvature.
\newblock {\em arXiv:2206.11796}, 2020.

\bibitem{CDNPSW}
Chris Connell, Xianzhe Dai, Jes\'us N\'{u}ñez-Zimbr\'on, Raquel Perales, Pablo
  Su\'arez-Serrato, and Guofang Wei.
\newblock Volume entropy and rigidity for {RCD}--spaces.
\newblock {\em In preparation}, 2023.

\bibitem{FF}
Herbert Federer and Wendell~H. Fleming.
\newblock Normal and integral currents.
\newblock {\em Ann. of Math. (2)}, 72:458--520, 1960.

\bibitem{gromov2014}
Misha Gromov.
\newblock Plateau-{S}tein manifolds.
\newblock {\em Cent. Eur. J. Math.}, 12(7):923--951, 2014.

\bibitem{HL}
Lan-Hsuan Huang and Dan~A. Lee.
\newblock Stability of the positive mass theorem for graphical hypersurfaces of
  {E}uclidean space.
\newblock {\em Comm. Math. Phys.}, 337(1):151--169, 2015.

\bibitem{HLP}
Lan-Hsuan Huang, Dan~A. Lee, and Raquel Perales.
\newblock Intrinsic flat convergence of points and applications to stability of
  the positive mass theorem.
\newblock {\em Ann. Henri Poincar\'{e}}, 23(7):2523--2543, 2022.

\bibitem{HLSa}
Lan-Hsuan Huang, Dan~A. Lee, and Christina Sormani.
\newblock Intrinsic flat stability of the positive mass theorem for graphical
  hypersurfaces of {E}uclidean space.
\newblock {\em J. Reine Angew. Math.}, 727:269--299, 2017.

\bibitem{HLSb}
Lan-Hsuan Huang, Dan~A. Lee, and Christina Sormani.
\newblock Corrigendum to: {I}ntrinsic flat stability of the positive mass
  theorem for graphical hypersurfaces of {E}uclidean space ({J}. {R}eine
  {A}ngew. {M}ath. 727 (2017), 269--299).
\newblock {\em J. Reine Angew. Math.}, 785:273--274, 2022.

\bibitem{Li}
Nan Li.
\newblock Lipschitz-volume rigidity in {A}lexandrov geometry.
\newblock {\em Adv. Math.}, 275:114--146, 2015.

\bibitem{LiWang}
Nan Li and Feng Wang.
\newblock Lipschitz-volume rigidity on limit spaces with {R}icci curvature
  bounded from below.
\newblock {\em Differential Geom. Appl.}, 35:50--55, 2014.

\bibitem{Res}
Ju.~G. Re\v{s}etnjak.
\newblock Liouville's conformal mapping theorem under minimal regularity
  hypotheses.
\newblock {\em Sibirsk. Mat. \v{Z}.}, 8:835--840, 1967.

\bibitem{SS}
Anna Sakovich and Christina Sormani.
\newblock Almost rigidity of the positive mass theorem for asymptotically
  hyperbolic manifolds with spherical symmetry.
\newblock {\em Gen. Relativity Gravitation}, 49(9):Paper No. 125, 26, 2017.

\bibitem{SY}
Richard Schoen and Shing~Tung Yau.
\newblock On the proof of the positive mass conjecture in general relativity.
\newblock {\em Comm. Math. Phys.}, 65(1):45--76, 1979.

\bibitem{sormani2017}
Christina Sormani.
\newblock {\em Scalar curvature and intrinsic flat convergence}, pages
  288--338.
\newblock Partial Differ. Equ. Meas. Theory. De Gruyter Open, Warsaw, 2017.

\bibitem{S:AA}
Christina Sormani.
\newblock Intrinsic flat {A}rzela-{A}scoli theorems.
\newblock {\em Comm. Anal. Geom.}, 26(6):1317--1373, 2018.

\bibitem{SW2}
Christina Sormani and Stefan Wenger.
\newblock The intrinsic flat distance between {R}iemannian manifolds and other
  integral current spaces.
\newblock {\em J. Differential Geom.}, 87(1):117--199, 2011.

\bibitem{W}
Edward Witten.
\newblock A new proof of the positive energy theorem.
\newblock {\em Comm. Math. Phys.}, 80(3):381--402, 1981.

\end{thebibliography}
